\documentclass[final,1p,times]{elsarticle}

\usepackage[colorlinks=true]{hyperref}
\usepackage{mathrsfs}

\usepackage{amsthm,amsmath,amssymb,mathrsfs,amsfonts,graphicx,graphics,latexsym,exscale,cmmib57,dsfont,bbm,amscd,ulem}
\usepackage{color,soul}
\newtheorem*{ose}{Oselede\v{c}'s multiplicative ergodic theorem}
\newtheorem*{CMlem}{Lemma}

\newtheorem{thm}{Theorem}[section]

\newtheorem{lem}[thm]{Lemma}
\newtheorem{cor}[thm]{Corollary}

\newcommand{\lam}[1]{\lambda_}

\newcommand{\bea}{\begin{eqnarray}}
\newcommand{\eea}{\end{eqnarray}}

\theoremstyle{definition}%
\newtheorem{rem}{Remark}
\newtheorem{defn}[thm]{Definition}
\newtheorem{example}[thm]{Example}

\def\bdot{\boldsymbol\cdot}
\def\bA{{\boldsymbol{A}}}
\def\bS{{\boldsymbol{S}}}
\def\X{{\boldsymbol{X}}}
\def\bxi{{\boldsymbol\xi}}

\def\bsigma{{\boldsymbol\sigma}}

\def\P{{\boldsymbol{P}}}

\numberwithin{equation}{section}

\journal{Ergodic Theory and Dynamical Systems}
\begin{document}

\begin{frontmatter}
\title{Finer filtration for matrix-valued cocycle based on Oselede\v{c}'s multiplicative ergodic theorem\tnoteref{t1}}

\author{Xiongping Dai}
\ead{xpdai@nju.edu.cn}


\address{Department of Mathematics, Nanjing University, Nanjing 210093, People's Republic of China}

\begin{abstract}
We consider a measurable matrix-valued cocycle $\bA\colon\mathbb{Z}_+\times X\rightarrow\mathbb{R}^{d\times d}$, driven by a measure-preserving transformation $T$ of a probability space $(X,\mathscr{F},\mu)$, with the integrability condition $\log^+\|\bA(1,\cdot)\|\in L^1(\mu)$. We show that for $\mu$-a.e. $x\in X$, if $\lim_{n\to\infty}\frac{1}{n}\log\|\bA(n,x)v\|=0$ for all $v\in\mathbb{R}^d\setminus\{0\}$, then the trajectory $\{\bA(n,x)v\}_{n=0}^\infty$ is far away from $0$ (i.e. $\limsup_{n\to\infty}\|\bA(n,x)v\|>0$) and there is some nonzero $v$ such that $\limsup_{n\to\infty}\|\bA(n,x)v\|\ge\|v\|$. This improves the classical multiplicative ergodic theorem of Oselede\v{c}. We here present an application to linear random processes to illustrate the importance.
\end{abstract}

\end{frontmatter}

\section{Introduction}\label{sec1}%
Throughout this paper, let $\mathbb{R}^{d\times d}$, $1\le d<\infty$, represent the space of all real $d\times d$ matrices, endowed with the matrix/operator norm $\pmb{\|}\cdot\pmb{\|}$ induced by an arbitrarily given (but not necessarily the usual Euclidean) vector norm $\|\cdot\|$ on $\mathbb{R}^d$. By $\mathbb{Z}_+$ we mean the set of all nonnegative integers. Let
$$T\colon X\rightarrow X$$
be a measure-preserving transformation of a probability space $(X,\mathscr{F},\mu)$, which is not necessarily ergodic. If $X$ is a Polish space and $\mathscr{F}$ is the Borel $\sigma$-field, then we call $(X,\mathscr{F},\mu)$ a Polish probability space. In this paper, we consider a measurable matrix-valued cocycle
\begin{equation}\label{eq1.1}
\bA\colon \mathbb{Z}_+\times X\rightarrow\mathbb{R}^{d\times d},
\end{equation}
driven by $T$; that is, it holds the following cocycle property:
\begin{equation*}
\bA(0,x)=\mathrm{I}_{d\times d},\quad\bA(m+n,x)=\bA(n,T^m(x))\bA(m,x)\quad \forall x\in X\textrm{ and }m,n\in\mathbb{Z}_+.
\end{equation*}
Here and in the sequel, $\mathrm{I}_{d\times d}$ stands for the unit matrix in $\mathbb{R}^{d\times d}$, and we will identify $B\in\mathbb{R}^{d\times d}$ with its induced linear transformation $u\mapsto Bu$ from $\mathbb{R}^d$ into itself.
Write
\begin{equation*}
\log^+\pmb{\|}\bA(1,x)\pmb{\|}:=\max\left\{0,\log\pmb{\|}\bA(1,x)\pmb{\|}\right\}.
\end{equation*}
Then the classical Oselede\v{c} multiplicative ergodic theorem may be stated as:

\begin{ose}[\cite{Ose}]
If $\log^+\pmb{\|}\bA(1,\cdot)\pmb{\|}\in L^1(\mu)$, then there exists $B\in\mathscr{F}$ with $T(B)\subseteq B$ and $\mu(B)=1$ such that:
\begin{enumerate}
\item[$(\mathrm{a})$] There is a measurable function $s\colon B\rightarrow\mathbb{N}$ with $s\circ T=s$.

\item[$(\mathrm{b})$] If $x$ belongs to $B$ there are $s(x)$ numbers $-\infty=\lambda_1(x)<\lambda_2(x)<\dotsm<\lambda_{s(x)}^{}(x)<\infty$.

\item[$(\mathrm{c})$] There are measurable linear subspaces of $\mathbb{R}^d$:
\begin{equation*}
\varnothing=V^{(0)}(x)\subset V^{(1)}(x)\subset\dotsm\subset V^{(s(x))}(x)=\mathbb{R}^d\quad \forall x\in B.
\end{equation*}
\item[$(\mathrm{d})$] If $x\in B$ and $1\le i\le s(x)$ then for all $v\in V^{(i)}(x)\setminus V^{(i-1)}(x)$,
\begin{equation*}
\lim_{n\to\infty}\frac{1}{n}\log\|\bA(n,x)v\|=\lambda_i(x).
\end{equation*}
\item[$(\mathrm{e})$] The function $\lambda_i(x)$ is defined and measurable on $\{x\,|\,s(x)\ge i\}$ and $\lambda_i(T(x))=\lambda_i(x)$ on this set.

\item[$(\mathrm{f})$] $\bA(1,x)V^{(i)}(x)\subseteq V^{(i)}(T(x))$ if $i\le s(x)$.
\end{enumerate}
\end{ose}

We note here that in our statement of Oselede\v{c}'s theorem above, we see $-\infty$ as the smallest Lyapunov exponent of $\bA$ at $x$ for the zero vector. So $s(x)\le d+1$ and sometimes $V^{(1)}(x)=\{0\}$. In Oselede\v{c}'s original statement (\cite{Ose} also \cite[Theorem~10.2]{Wal82}), ones do not care the zero vector and then $V^{(0)}(x)=\{0\}$ for this case.

\subsection{Motivation}
With the above Oselede\v{c} multiplicative ergodic theorem in mind, we will further consider the following two important questions:
\begin{itemize}
\item[$(1)$] Does it hold that for $2\le i\le s(x)$,
$$
\limsup_{n\to\infty}e^{-\lambda_i(x)n}\|\bA(n,x)v\|>0\quad \left(\textrm{or }\liminf_{n\to\infty}e^{-\lambda_i(x)n}\|\bA(n,x)v\|>0\right)
$$
for all $v\in V^{(i)}(x)\setminus V^{(i-1)}(x)$?

\item[$(2)$] Does it hold that
$$
\limsup_{n\to\infty}e^{-\lambda_i(x)n}\pmb{\|}\bA(n,x)|V^{(i)}(x)\pmb{\|}\ge1
$$
and 
$$
\limsup_{n\to\infty}e^{-\lambda_i(x)n}\|\bA(n,x)v\|\ge\|v\|\quad v\in V^{(i)}(x)\setminus V^{(i-1)}(x)
$$
for $2\le i\le s(x)$?
\end{itemize}
Even in the invertible case, the above two properties are already beyond the improved multiplicative ergodic theorem of Froyland, LLoyd and Quas \cite{FLQ}.

These properties (1) and (2) are important to the stabilizability problem of linear switching dynamical systems; see, e.g., \cite{Dai-JDE, DHX-IEEE} and an application presented later.

\subsection{Main statement}

In this paper, our aim is to further employ the filtration given by Oselede\v{c}'s theorem. Our main theorem is stated as follows.

\begin{thm}[MET]\label{thm1.1}
Let $T$ be a measure-preserving transformation of a probability space $(X,\mathscr{F},\mu)$ and $\bA\colon\mathbb{Z}_+\times X\rightarrow\mathbb{R}^{d\times d}$ measurable such that $\log^+\pmb{\|}\bA(1,\cdot)\pmb{\|}\in L^1(\mu)$. Then there exists a set $B\in\mathscr{F}$ with $T(B)\subseteq B$ and $\mu(B)=1$ such that:
\begin{enumerate}
\item[$(\mathrm{a})$] There is a measurable function $s\colon B\rightarrow\mathbb{N}$ with $s\circ T=s$.

\item[$(\mathrm{b})$] If $x$ belongs to $B$ there are $s(x)$ numbers $-\infty=\lambda_1(x)<\lambda_2(x)<\dotsm<\lambda_{s(x)}(x)<\infty$.

\item[$(\mathrm{c})$] There are measurable linear subspaces of $\mathbb{R}^d$:
\begin{equation*}
\varnothing=V^{(0)}(x)\subset V^{(1)}(x)\subset\dotsm\subset V^{(s(x))}(x)=\mathbb{R}^d\quad \forall x\in B.
\end{equation*}
\item[$(\mathrm{d})$] If $x$ belongs to $B$, then
\begin{enumerate}
\item[$(\mathrm{i})$] for $1\le i\le s(x)$,
\begin{equation*}
\lim_{n\to\infty}\frac{1}{n}\log\|\bA(n,x)v\|=\lambda_i(x)\quad \forall v\in V^{(i)}(x)\setminus V^{(i-1)}(x);
\end{equation*}
\item[$(\mathrm{ii})$] for $2\le i\le s(x)$,
\begin{gather*}
\limsup_{n\to\infty}e^{-\lambda_i(x)n}\|\bA(n,x)v\|>0\quad \forall v\in V^{(i)}(x)\setminus V^{(i-1)}(x)\\
\intertext{and}
\limsup_{n\to\infty}e^{-\lambda_i(x)n}\pmb{\|}\bA(n,x)|V^{(i)}(x)\pmb{\|}\ge1;
\end{gather*}
\item[$(\mathrm{iii})$] if we additionally let $T$ be measure-preserving on a Polish probability space $(X,\mathscr{F},\mu)$, then for $2\le i\le s(x)$ one can find some $v_i\in V^{(i)}(x)\setminus V^{(i-1)}(x)$ such that
\begin{equation*}
\limsup_{n\to\infty}e^{-\lambda_i(x)n}\|\bA(n,x)v_i\|\ge\|v_i\|.
\end{equation*}
\end{enumerate}

\item[$(\mathrm{e})$] The function $\lambda_i(x)$ is defined and measurable on $\{x\,|\,s(x)\ge i\}$ and $\lambda_i^{}(T(x))=\lambda_i(x)$ on this set.

\item[$(\mathrm{f})$] For any $x\in B$ and all $1\le i\le s(x)$,
\begin{enumerate}
\item[$(\mathrm{i})$] $\bA(1,x)V^{(i)}(x)\subseteq V^{(i)}(T(x))$ and
\item[$(\mathrm{ii})$] $\dim V^{(i)}(T(x))=\dim V^{(i)}(x)$.
\end{enumerate}
\end{enumerate}
\end{thm}

This result improves the classical Oselede\v{c} multiplicative ergodic theorem and its recent improved version of Froyland, LLoyd and Quas~\cite[Theorem~4.1]{FLQ} by the items (d)-(ii), (iii) and (f)-(ii). If $\mu$ is ergodic, then $s(x)$ and $\lambda_i(x)$, for $1\le i\le s(x)$, all are constants mod 0 under the sense of the probability measure $\mu$.

\begin{rem}\label{rem1}
Although for each $v\in V^{(i)}(x)\setminus V^{(i-1)}(x)$ for $2\le i\le s(x)$, $e^{-\lambda_{i}(x)n}\bA(n,x)v$ has the Lyapunov exponent zero, yet $e^{-\lambda_{i}(x)n}\bA(n,x)v$ might converge asymptotically to $0$ for some base points $x\in B$. Even if we additionally impose the condition that $\bA$ is \textit{uniformly product-bounded},\footnote{It should be noted that since here $\bigsqcup_{n\ge0}\{\bA(n,x)\,|\,x\in X\}$ does not need to be a semigroup under the matrix multiplication, we cannot choose a ``good'' norm $\|\cdot\|_*$ on $\mathbb{R}^d$ so that $\pmb{\|}\bA(1,x)\pmb{\|}_*\le1$ for all $x\in X$ under this uniform product-boundedness condition.} i.e.,
$$
\exists\,\beta\ge1\textrm{ such that }\pmb{\|}\bA(n,x)\pmb{\|}\le\beta\quad\forall n\in\mathbb{Z}_+\textrm{ and }x\in X,
$$
this phenomenon still possibly happens (cf.~Example~\ref{exa2.5} in Section~\ref{sec2.1}) for $i=s(x)$. Thus the properties (d)-(ii) and (iii) of Theorem~\ref{thm1.1} are the main nontrivial new points here; and
our theorem is more finer than Oselede\v{c}'s theorem.

On the other hand, without the uniform product-boundedness condition our Theorem~\ref{thm1.1}
is also beyond the situation of Morris~\cite[Theorems~2.1 and 2.2]{Mor-AM}.
\end{rem}

\begin{rem}\label{rem2}
For any $x\in B$ and $1\le i\le s(x)$, we write
$$
\widehat{V}^{(i)}(x)=V^{(i)}(x)\cap V^{(i-1)}(x)^\perp.
$$
Now for any $v\in\widehat{V}^{(i)}(x)$ and $n\ge1$, let $\widehat{\bA}^{(i)}(n,x)v$ be the projection of $\bA(n,x)v$ onto $\widehat{V}^{(i)}(T^n(x))$. It is easy to see that
$\widehat{\bA}^{(i)}(m+n,x)v=\widehat{\bA}^{(i)}(n,T^m(x))\widehat{\bA}^{(i)}(m,x)v$. Thus we can define the linear transformations
$$
\widehat{\bA}^{(i)}(n,x)\colon\widehat{V}^{(i)}(x)\rightarrow\widehat{V}^{(i)}(T^n(x)).
$$
Based on the property (f) of Theorem~\ref{thm1.1} we have
$$\mathbb{R}^d=\widehat{V}^{(1)}(x)\oplus\dotsm\oplus\widehat{V}^{(s(x))}(x).$$
Moreover from the property (f) of Theorem~\ref{thm1.1},
\begin{equation*}
\widehat{\bA}(n,x)\colon\widehat{V}^{(1)}(x)\oplus\dotsm\oplus\widehat{V}^{(s(x))}(x)\rightarrow\widehat{V}^{(1)}(T^n(x))\oplus\dotsm\oplus\widehat{V}^{(s(x))}(T^n(x)),
\end{equation*}
defined by
$$v=v_1+\dotsm+v_{s(x)}\mapsto\widehat{\bA}(n,x)v=\widehat{\bA}^{(1)}(n,x)v_1+\dotsm+\widehat{\bA}^{(s(x))}(n,x)v_{s(x)},$$
induces a natural cocycle $\widehat{\bA}\colon\mathbb{Z}_+\times X\rightarrow\mathbb{R}^{d\times d}$, for $\mu$-a.e. $x\in X$ which has the same Lyapunov exponents $-\infty=\lambda_1(x)<\lambda_2(x)<\dotsm<\lambda_{s(x)}(x)<\infty$ associated to the $\widehat{\bA}$-invariant direct decomposition $\widehat{V}^{(1)}(x)\oplus\dotsm\oplus\widehat{V}^{(s(x))}(x)$ and the same stability as $\bA$ according to Liao's version of the multiplicative ergodic theorem (cf.~\cite{Liao} and also see \cite{Dai09}).

Since our driving system $T$ is not necessarily to be invertible, there is no such an $\bA$-invariant direct decomposition of $\mathbb{R}^d$ corresponding to the Lyapunov exponents, the induced cocycle $\widehat{\bA}$ should be useful.
\end{rem}

\begin{rem}\label{rem3}
For the filtration given by Theorem~\ref{thm1.1}, it is possible that
\begin{equation*}
\liminf_{n\to\infty}e^{-\lambda_i(x)n}\pmb{\|}\bA(n,x)|\widehat{V}^{(i)}(x)\pmb{\|}_{\min}=0,\quad 2\le i\le s(x),
\end{equation*}
for $\mu$-a.e. $x\in X$; for example, $\bA(n,x)\equiv\left[\begin{smallmatrix}1&n\\0&1\end{smallmatrix}\right]$ for all $x\in X$ and $n\ge1$. For which $\lambda_2(x)\equiv0$ is a Lyapunov exponent of multiplicity 2 and $V^{(1)}(x)\equiv\{0\}, \widehat{V}^{(2)}(x)\equiv \mathbb{R}^2$.
\end{rem}

\subsection{Outline}
The remains of this paper will be simply organized as follows. We shall prove our main result Theorem~\ref{thm1.1} in Section~\ref{sec2} based on Oselede\v{c}'s theorem and its improved version of Froyland, LLoyd and Quas~\cite[Theorem~4.1]{FLQ}, and a series of lemmas. We will present an application to asymptotical stability of linear random processes (Theorem~\ref{thm3.5}) in Sections~\ref{sec3}.

\section{Finer filtration of a linear cocycle}\label{sec2}
This section will be devoted to proving our main results Theorem~\ref{thm1.1} stated in Section~\ref{sec1}. Besides Oselede\v{c}'s multiplicative ergodic theorem stated in Section~\ref{sec1}, we will need a series of independently interesting lemmas here including the measurability of stable manifolds and the non-oscillatory behavior of subadditive random process.

\subsection{The measurability of stable manifolds}\label{sec2.1}
Let $T\colon X\rightarrow X$ be a measurable transformation on a measurable space $(X,\mathscr{F})$, where $X$ does not need to be a topological space, and we simply write $T^t(x)=t\bdot x$ for all $t\in\mathbb{Z}_+$. From now on, let $\bA\colon\mathbb{Z}_+\times X\rightarrow\mathbb{R}^{d\times d}$ be a measurable matrix-valued cocycle driven by $T$, which is not necessarily to satisfy the integrability condition.

For an arbitrary integer $p$ with $0\le p\le d$, by $\mathscr{G}(p,\mathbb{R}^d)$ we denote the set of all $p$-dimensional subspaces of $\mathbb{R}^d$ and
set
\begin{equation*}
\mathscr{G}(\mathbb{R}^d)=\bigsqcup_{p=0}^d\mathscr{G}(p,\mathbb{R}^d),
\end{equation*}
where $\bigsqcup$ means the disjoint union. We equip $\mathscr{G}(\mathbb{R}^d)$ with the standard compact topology induced by the Hausdorff metric $\mathrm{D_H}(\cdot,\cdot)$, i.e., for any $V,W\in\mathscr{G}(\mathbb{R}^d)$,
\begin{equation*}
\mathrm{D_H}(V,W)=\max\left\{\sup_{v\in V^\sharp}\min_{w\in W^\sharp}\|v-w\|,\sup_{w\in W^\sharp}\min_{v\in V^\sharp}\|w-v\|\right\}
\end{equation*}
where $V^\sharp=\{v\in V\colon\|v\|=1\}$ if $\dim V>0$ and $V^\sharp=\{0\}$ if $\dim V=0$.
Here $\|\cdot\|$ is the given vector norm on $\mathbb{R}^d$ as before.

For any $x\in X$, we set
\begin{equation}\label{eq2.1}
V^s(x)=\left\{u\in\mathbb{R}^d\colon\|\bA(t,x)u\|\to0\textrm{ as }t\to\infty\right\}.
\end{equation}
It is a linear subspace of $\mathbb{R}^d$ with the invariance
\begin{equation}
\bA(t,x)V^s(x)\subseteq V^s(t\bdot x)\quad \forall t\in\mathbb{Z}_+
\end{equation}
and we call it the ``stable manifold/direction" of $\bA$ over the driving base point $x$.

As to be shown by Example~\ref{exa2.5} below, $V^s(x)$ is not necessarily equal to the exponentially stable space
\begin{equation}
E^s(x)=\left\{u\in\mathbb{R}^d\colon\lim_{n\to\infty}\frac{1}{n}\log\|\bA(n,x)u\|<0\right\}.
\end{equation}
Hence the measurability of $V^s(x)$ with respect to the base point $x$ is not an obvious consequence of the classical multiplicative ergodic theorem.

To prove Theorem~\ref{thm1.1} stated in Section~\ref{sec1}, we need to use the following preliminary result on the regularity of $V^s(x)$ respecting $x$ in the natural measure-theoretic sense.

\begin{thm}\label{thm2.1}
Under the situation above and given any probability measure $\mu$ on $(X,\mathscr{F})$, there exists an $\mathscr{F}$-measurable function $\eta\colon X\rightarrow\mathscr{G}(\mathbb{R}^d)$ such that
\begin{equation*}
\eta(x)=V^s(x)\quad\textrm{and}\quad\dim V^s(x)\ge\dim V^s(n\bdot x)\ \forall n\ge1
\end{equation*}
for $\mu$-a.e. $x\in X$.
\end{thm}

To prove this theorem, we first introduce some classical results in measure theory. For a probability measure $\mu$ on the measurable space $(X,\mathscr{F})$, let $\mathscr{F}_\mu$ be the ``completion" of $\mathscr{F}$ respecting $\mu$, as the $\sigma$-field
$\mathscr{F}_\mu=\sigma(\mathscr{F}\cup\mathscr{N}_\mu)$, where $\mathscr{N}_\mu$ denotes the class of all subsets of arbitrary $\mu$-null sets in $\mathscr{F}$. Then, the ``universal completion'' of $\mathscr{F}$ is defined as the $\sigma$-field $\mathscr{F}^*=\bigcap_\mu\mathscr{F}_\mu$, where the intersection extends over all probability measures $\mu$ on $\mathscr{F}$.

Then not only $\mu$ has a unique extension from $\mathscr{F}$ to the $\sigma$-field $\mathscr{F}_\mu$ but also $T$ is a measure-preserving transformation of $(X,\mathscr{F}_\mu,\mu)$. This completion enables us using some classical results in measure theory.

\begin{lem}[Completion]\label{lem2.2}
Let $(Y,\mathscr{B}_Y)$ be a Borel measurable space and $\mu$ a probability measure on $(X,\mathscr{F})$. Then, a function $f\colon X\rightarrow Y$
is $\mathscr{F}_\mu/\mathscr{B}_Y$-measurable if and only if there exists some $\mathscr{F}/\mathscr{B}_Y$-measurable function $g\colon X\rightarrow Y$ such that $f(x)=g(x)$ for $\mu$-a.e. $x\in X$.
\end{lem}

This result is well known and can be found in many textbooks on real analysis and probability, for example, in \cite[Lemma~1.25]{Kal}.
The following is another classical result needed.

\begin{lem}[{Projection and Sections; Lusin, Choquet, Meyer; see~\cite[Theorem~A1.4]{Kal}}]\label{lem2.3}
Let $(Y,\mathscr{B}_Y)$ be an arbitrary Borel measurable space and $\pi\colon X\times Y\rightarrow X$ the canonical projection defined by $(x,y)\mapsto x$. Then for any $B\in\mathscr{F}\otimes\mathscr{B}_Y$,
\begin{enumerate}
\item[$\mathrm{(i)}$] $\pi(B)$ belongs to $\mathscr{F}^*$;

\item[$\mathrm{(ii)}$] for any probability measure $\mu$ on $\mathscr{F}$, there exists a $\mathscr{F}_\mu$-measurable element $\eta\colon X\rightarrow Y$, which is such that $(x,\eta(x))\in B$ for $\mu$-a.e. $x\in\pi(B)$.
\end{enumerate}
\end{lem}

As a result, we can obtain the following from Lemma~\ref{lem2.3}.

\begin{lem}\label{lem2.4}
Let $(Y,\mathscr{B}_Y)$ be an arbitrary Borel measurable space, $\mu$ a probability measure on $(X,\mathscr{F})$, and $\pi\colon X\times Y\rightarrow X$ the canonical projection. Then for any $B\in\mathscr{F}_\mu\otimes\mathscr{B}_Y$,
\begin{enumerate}
\item[$\mathrm{(i)}$] $\pi(B)$ belongs to $\mathscr{F}_\mu$;

\item[$\mathrm{(ii)}$] there exists a $\mathscr{F}_\mu$-measurable element $\eta\colon X\rightarrow Y$ such that $(x,\eta(x))\in B$ for $\mu$-a.e. $x\in\pi(B)$.
\end{enumerate}
\end{lem}

\begin{proof}
This result follows from Lemma~\ref{lem2.3} with $\mathscr{F}_\mu$ in place of $\mathscr{F}$ and the simple fact that $(\mathscr{F}_\mu)_\mu=\mathscr{F}_\mu$.
\end{proof}

We think of $\overline{\mathbb{R}}_+=[0,+\infty)\cup\{+\infty\}$ as the one-point compactification of $\mathbb{R}_+$; then $\mathbb{R}_+$ is an open subset of $\overline{\mathbb{R}}_+$. We are now ready to complete the proof of Theorem~\ref{thm2.1}.

\begin{proof}[Proof of Theorem~\ref{thm2.1}]
Based on the measurable cocycle $\bA$, let $f_n\colon X\times\mathscr{G}(\mathbb{R}^d)\rightarrow\overline{\mathbb{R}}_+$ be defined as $f_n(x,L)=\pmb{\|}\bA(n,x)|L\pmb{\|}$ for all $(x,L)\in X\times\mathscr{G}(\mathbb{R}^d)$, for $n=1,2,\dotsc$. Clearly, $f_n$ is naturally measurable. Set
\begin{equation*}
\bar{f}(x,L)=\limsup_{n\to\infty}f_n(x,L)\quad \forall (x,L)\in X\times\mathscr{G}(\mathbb{R}^d).
\end{equation*}
Then $\bar{f}\colon X\times\mathscr{G}(\mathbb{R}^d)\rightarrow\overline{\mathbb{R}}_+$ is also naturally measurable. Let
\begin{equation*}
{\ae}=\left\{\bar{f}=0\right\}={\bigsqcup}_{x\in X}{\ae}_x,\quad \textrm{where }\ae_x=\left\{L\colon\bar{f}(x,L)=0\right\}.
\end{equation*}
Then $\ae$ is a $\mathscr{F}\otimes\mathscr{B}_{\mathscr{G}(\mathbb{R}^d)}$-set and $\ae_x$ is a $\mathscr{B}_{\mathscr{G}(\mathbb{R}^d)}$-set.

We note that $V^s(x)\in\ae_x$, $L\subseteq V^s(x)$ for $L\in\ae_x$, and $\dim V^s(x)=\max\{\dim L\colon L\in\ae_x\}$ for each $x\in X$.
As the function $(x,L)\mapsto\dim L$ is measurable from $X\times\mathscr{G}(\mathbb{R}^d)$ into the discrete-topological space $\{0,1,\dotsc,d\}$,
from repeatedly using the item (i) of Lemma~\ref{lem2.3} we can easily obtain that $B:=\{(x,V^s(x))\,|\,x\in X\}$ is a $\mathscr{F}^*/\mathscr{B}_{\mathscr{G}(\mathbb{R}^d)}$-set.

Now from Lemma \ref{lem2.4} with $Y=\mathscr{G}(\mathbb{R}^d)$, it follows that there is an $\hat{\ae}\colon X\rightarrow\mathscr{G}(\mathbb{R}^d)$ that is $\mathscr{F}_\mu/\mathscr{B}_{\mathscr{G}(\mathbb{R}^d)}$-measurable such that $\hat{\ae}(x)=V^s(x)$ for $\mu$-a.e. $x\in X$. Then from Lemma~\ref{lem2.2}, one can find some $\mathscr{F}/\mathscr{B}_{\mathscr{G}(\mathbb{R}^d)}$-measurable $\eta\colon X\rightarrow\mathscr{G}(\mathbb{R}^d)$ satisfying the requirement of Theorem~\ref{thm2.1}.

Finally we note that $\dim (V^s(x))^\perp\le\dim (V^s(t\bdot x))^\perp$ and hence $\dim V^s(x)\ge\dim V^s(t\bdot x)$ for all $t\ge1$ and $x\in X$.

This completes the proof of Theorem~\ref{thm2.1}.
\end{proof}

If $V^s(x)=E^s(x)$ mod 0, then for Theorem~\ref{thm2.1} there is nothing needed to prove from the property (c) of Oselede\v{c}'s multiplicative ergodic theorem. The following example shows that the proof of Theorem~\ref{thm2.1} is nontrivial even though under an additional condition---the uniform product-boundedness.

Let $\varSigma_2^+=\{\sigma\colon\mathbb{Z}_+\rightarrow\{0,1\}\}$ be the one-sided symbolic space and $$T\colon\sigma=(\sigma(n))_{n\ge0}\mapsto\sigma(\bdot+1)=(\sigma(n+1))_{n\ge0}$$
the classical one-sided shift transformation on $\varSigma_2^+$.

\begin{example}\label{exa2.5}
Driven by the shift transformation $T$, let $\bA\colon\mathbb{Z}_+\times\varSigma_2^+\rightarrow\mathbb{R}^{2\times 2}$ be naturally induced by $A=\{A_0,\, A_1\}$ with
 \begin{equation*}
   A_0=\left[\begin{matrix}1 & 0\\ 0&
   1\end{matrix}\right]\quad \textrm{and} \quad A_1=\left[\begin{matrix}\frac{1}{2} & 0\\ 0&
   1\end{matrix}\right],\quad \textrm{i.e.},\ \bA(n,\sigma)=A_{\sigma(n-1)}\dotsm A_{\sigma(0)}\ \forall n\ge1.
 \end{equation*}
Clearly $\bA$ is uniformly product-bounded. By induction we now construct a switching sequence $\bsigma$ over which $\bA(t,\bsigma)v$ converges to $0$
but not exponentially fast as $t\to\infty$ for some initial state $v\not=0$.

To this end, for a word $w=(i_0,\dotsc, i_{k-1})\in \{0,1\}^k$, let
$|w|=k$ denote the length of the word $w$ and $\mathrm{O}_k$ stand for the $k$-length
word consisting of $k$ numbers of $0$, i.e, $\mathrm{O}_k=(0,\dotsc,0)\in\{0,1\}^k$. For any pair of words $u=(u_1,\dotsc,u_k)$ and $w=(w_1,\dotsc,w_m)$, we set
\begin{equation*}
uw=(u_1,\dotsc,u_k,w_1,\dotsc,w_m)\in\{0,1\}^{k+m}
\end{equation*}
Let $\sigma_1=(1)$,
$\sigma_2=(\sigma_1\mathrm{O}_{|\sigma_1|^2}\sigma_1)$. Inductively, for $n\geq 2 $, let
\begin{equation*}
\sigma_{n}=(\sigma_{n-1}\mathrm{O}_{|\sigma_{n-1}|^2}\sigma_{n-1})\quad \textrm{and}\quad
\bsigma=\lim_{n\to\infty}\sigma_n.
\end{equation*}
A routine check shows that for $v=\left(\begin{matrix}1\\0\end{matrix}\right)\in\mathbb{R}^2$,
\begin{gather*}
\lim_{n\to\infty}\|A_{\bsigma(n-1)}\dotsm
A_{\bsigma(0)}v\|=0\\
\intertext{and}
\limsup_{n\to\infty}\frac{1}{n}\log \|A_{\bsigma(n-1)}\dotsm
A_{\bsigma(0)}v\|=\lim_{n\to\infty}\frac{n}{n^2}\log\frac{1}{2}=0.
\end{gather*}
So the convergence is not exponentially fast based on the switching sequence $\bsigma$.
\end{example}

This completes the construction of Example~\ref{exa2.5}.
\subsection{Non-oscillatory behavior of a subadditive random process}\label{sec2.2}
To prove our Theorem~\ref{thm1.1}, we will need a result similar to Giles Atkinson's theorem on additive cocycles~\cite{Atk}. Atkinson's theorem (together with a result of K.~Schmidt) asserts the following.

\begin{lem}[Atkinson]\label{lem2.6}
If $T\colon(X,\mathscr{F},\mu)\rightarrow(X,\mathscr{F},\mu)$ is an ergodic measure-preserving automorphism and $f\colon X\rightarrow\mathbb{R}$ is an integrable function with $\int_Xfd\mu=0$, then for $\mu$-a.e. $x\in X$ the sum $\sum_{k=0}^{n-1}f(T^kx)$ returns arbitrarily close to zero infinitely often.
\end{lem}

The following similar lemma has not previously been formally published, but arose in discussion between Dr. Vaughn Climenhaga and Dr. Ian Morris on the MathOverflow internet forum, where their proof is adapted from G.~Atkinson's argument.\footnote{Cf.~http://mathoverflow.net/questions/70676/ for the details. We would like to thank those authors for agreeing to the inclusion of this lemma in the present document.}

\begin{CMlem}[Climenhaga and Morris]
Let $T$ be an ergodic measure-preserving transformation of a probability space $(X,\mathscr{F},\mu)$, and let $(f_n)_{n\ge1}$ be a sequence of integrable functions from $X$ to $\mathbb{R}$, which satisfies the subadditivity relation:
\begin{equation*}
f_{n+m}(x)\le f_n(T^mx)+f_m(x)\quad \textrm{for }\mu\textrm{-a.e. }x\in X\textrm{ and }n,m\ge1.
\end{equation*}
Suppose that ${\lim}_{n\to\infty}f_n(x)=-\infty$ for $\mu$-a.e. $x\in X$. Then
${\lim}_{n\to\infty}n^{-1}\int_Xf_n(x)d\mu(x)<0$.
\end{CMlem}

Now we will introduce a general result of independent interest, which is a generalization of the above lemma and \cite[Theorem~2.4]{Dai-JDE}.

\begin{thm}\label{thm2.7}
Let $T$ be a measure-preserving, not necessarily ergodic, transformation of a probability space $(X,\mathscr{F},\mu)$, and let $(f_n)_{n\ge1}$ be a sequence of measurable functions from $X$ to $\mathbb{R}\cup\{-\infty\}$ with $f_1^+\in L^1(\mu)$, which satisfies the subadditivity relation:
\begin{equation*}
f_{n+m}(x)\le f_n(T^mx)+f_m(x)\quad \textrm{for }\mu\textrm{-a.e. }x\in X\textrm{ and }n,m\ge1.
\end{equation*}
Let $F(x)={\limsup}_{n\to\infty}f_n(x)$ for $x\in X$.
Then the symmetric difference
\begin{equation*}
\left\{x\in X\,|\,F(x)<0\right\}\vartriangle\left\{x\in X\,|\,{\lim}_{n\to\infty}n^{-1}f_n(x)<0\right\}
\end{equation*}
has $\mu$-measure $0$.
\end{thm}

\begin{proof}
By the Kingman subadditive ergodic theorem (cf.~\cite[Theorem~10.1]{Wal82}), there exists a measurable function $f\colon X\rightarrow\mathbb{R}\cup\{-\infty\}$ such that
\begin{equation*}
\lim_{n\to\infty}\frac{1}{n}f_n(x)=f(x)\quad\textrm{for }\mu\textrm{-a.e. }x\in X.
\end{equation*}
Let $\Lambda=\{x\in X\,|\,F(x)<0\}$. It is a measurable subset of $X$, since $F(x)$ is measurable.
Then from $\Lambda\supseteq\{x\in X\,|\,f(x)<0\}$, our task is to show that
\begin{equation}\label{eq2.4}
\mu(\{x\in X\,|\,f(x)<0\})\ge\mu(\Lambda).
\end{equation}
Without loss of generality, assume $\mu(\Lambda)>0$; otherwise we need to prove nothing.

Let $\varepsilon>0$ be arbitrarily given with $\varepsilon\ll\mu(\Lambda)$.
Because $F(x)<0$ for all $x\in\Lambda$, we can find a constant $\alpha>0$ for which $\Lambda_1=\{x\in\Lambda\,|\,F(x)\le-2\alpha\}$ is an $\mathscr{F}$-set such that $\mu(\Lambda_1)\ge\mu(\Lambda)-\frac{\varepsilon}{3}$. Since $\varepsilon$ is arbitrary, to prove (\ref{eq2.4}) it is sufficient to show that
\begin{equation}\label{eq2.5}
\mu(\{x\in X\,|\,f(x)<0\})>\mu(\Lambda)-\varepsilon.
\end{equation}

To this end, given $x\in\Lambda_1$, let $M_x=\{n\,|\,f_n(x)\ge-\alpha\}$, and observe that $M_x$ is only finite for all $x\in\Lambda_1$. Thus writing $A_n=\{x\in \Lambda_1\,|\,\#(M_x)<n\}$ where $\#$ stands for the cardinality of a set, we see that $A_n$ is an $\mathscr{F}$-set and that there exists an integer $N>1$ such that $\mu(A_N)\ge\mu(\Lambda_1)-\frac{\varepsilon}{3}$.

From the Birkhoff ergodic theorem, we have that
\begin{equation*}
\liminf_{n\to\infty}\frac{1}{n}\sum_{j=0}^{n-1}1_{A_N}(T^j(x))=1_{A_N}^*(x)\ \textrm{a.e.}\quad \textrm{and}\quad\int_X1_{A_N}^*(x)d\mu(x)=\mu(A_{N}).
\end{equation*}
Let $X_1=\{x\in X\,|\,1_{A_N}^*(x)>0\}$. Since $0\le1_{A_N}^*(x)\le1$, we have $\mu(X_1)\ge\mu(A_N)$. Thus to prove (\ref{eq2.5}), we need only prove that
$f(x)<0$ for all $x\in X_1$.

For that, we fix such an $x\in X_1$ from now on. Since $A_n\subseteq A_{n+1}$ and $1_{A_n}^*\le1_{A_{n+1}}^*$ mod 0 for all $n\ge1$,  we can take an integer $K>N$ such that $1_{A_K}^*(x)>\frac{1}{K}$. Write $L_x=\{k\ge0\,|\,T^k(x)\in A_K\}$. So,
\begin{equation}\label{eq2.6}
\#(L_x\cap[1,n])\ge\frac{n}{K}
\end{equation}
for all sufficiently large $n$.

Let $k_0$ be the smallest element of $L_x$ and define $f_0\equiv 0$ mod 0. We define integers $k_i\in L_x$, for $i=1,2,\dotsm$, recursively with the property that
\begin{equation}\label{eq2.7}
f_{k_i}(x)<f_{k_0}(x)-i\alpha,
\end{equation}
as follows. Let $J_i$ be the $K$ smallest elements of $L_x\cap(k_i,\infty)$. Because $T^{k_i}(x)\in A_K$, there exists $k_{i+1}\in J_i$ such that $k_{i+1}-k_i\not\in M_{T^{k_i}(x)}$. In particular, we have
\begin{equation}
f_{k_{i+1}-k_i}(T^{k_i}(x))<-\alpha.
\end{equation}
Now subadditivity gives
\begin{equation*}
f_{k_{i+1}}(x)\le f_{k_i}(x)+f_{k_{i+1}-k_i}(T^{k_i}(x))<f_{k_0}(x)-(i+1)\alpha.
\end{equation*}
The next observation to make is that by (\ref{eq2.6}), we can obtain
\begin{equation}
k_i\le K^2i\quad \textrm{for all sufficiently large }i
\end{equation}
by considering $n=K^2i$ in (\ref{eq2.6}).
Thus by (\ref{eq2.7}) we have
\begin{equation*}
\frac{1}{k_i}f_{k_i}(x)\le\frac{1}{K^2i}\left(f_{k_0}(x)-i\alpha\right),
\end{equation*}
and letting $i\to\infty$ we obtain $f(x)\le-\frac{\alpha}{K^2}$, as desired.

This completes the proof of Theorem~\ref{thm2.7}.
\end{proof}

\begin{rem}
We note that checking $F(x)<0$ in Theorem~\ref{thm2.7} is relatively easier than checking $\lim_{n\to\infty}f_n(x)=-\infty$. In addition, it should be noted here that since $(X,\mathscr{F},\mu)$ is not necessarily a Lebesgue space, Rohlin's ergodic decomposition theorem does not work for $T$ here and then Theorem~\ref{thm2.7} is not a corollary of the above lemma proposed by Ian Morris. Our proof is mainly adapted from that of \cite[Theorem~2.4]{Dai-JDE} and V.~Climenhaga.
\end{rem}

\begin{rem}\label{rem5}
In Lemma~\ref{lem2.6}, the condition of $f$ belonging to $L^1(\mu)$ is a technical obstruction for us to use Atkinson's theorem; for example, letting $A\colon X\rightarrow \mathbb{R}^{d\times d}$ being measurable with $\sup_{x\in X}\pmb{\|}A(x)\pmb{\|}<\infty$ then the characteristic function $\varphi(x,v)=\log\|A(x)v\|$ defined on the trivial unit-vector bundle $X\times S^{d-1}$ is not necessarily integrable but $\varphi^+(x,v)=\log^+\|A(x)v\|$ and hence $\varphi^+$ is integrable, as in the proof of Theorem~\ref{thm2.13} later.
\end{rem}
\subsection{The trajectory starting from nonstable direction is far away from zero}\label{sec2.3}

Let $T\colon X\rightarrow X$ be a measure-preserving transformation on a probability space $(X,\mathscr{F},\mu)$ and $\bA\colon\mathbb{Z}_+\times X\rightarrow\mathbb{R}^{d\times d}$ a measurable cocycle driven by $T$, where $\mu$ is not necessarily to be ergodic with respect to $T$.
For any $x\in X$, as in (\ref{eq2.1}) we set
\begin{equation}\label{eq2.10}
V^s(x)=\left\{v\in\mathbb{R}^d\colon\|\bA(t,x)v\|\to0\textrm{ as }t\to\infty\right\}.
\end{equation}
Clearly $\bA(t,x)V^s(x)\subseteq V^s(T^t(x))$ for all $t\ge1$. From Theorem~\ref{thm2.1}, there is no loss of generality in assuming
\begin{equation*}
X\ni x\mapsto V^s(x)\in\mathscr{G}(\mathbb{R}^d)
\end{equation*}
is a measurable function, replacing $X$ by some $T$-invariant $\mathscr{B}$-set of $\mu$-measure $1$ if necessary.

We will utilize the following simple result.

\begin{lem}\label{lem2.8}
For any linear subspace $L\subseteq\mathbb{R}^d$ and  $x\in X$, the following statements are equivalent to each other:
\begin{enumerate}
\item[$\mathrm{(a)}$] $\lim_{t\to\infty}\pmb{\|}\bA(t,x)|L\pmb{\|}=0$.

\item[$\mathrm{(b)}$] $\lim_{t\to\infty}\|\bA(t,x)v\|=0$ for all $v\in L$.
\end{enumerate}
\end{lem}

In the following theorem the new element needed to be proved is only the property (\ref{eq2.13}) from the viewpoint of Oselede\v{c}'s multiplicative ergodic theorem.

\begin{thm}\label{thm2.9}
Let $\bA\colon \mathbb{Z}_+\times X\rightarrow\mathbb{R}^{d\times d}$ be measurable such that $\log^+\pmb{\|}\bA(1,\cdot)\pmb{\|}\in L^1(\mu)$.
Then there exists a set $B^\prime\in\mathscr{F}$ with $T^t(B^\prime)\subseteq B^\prime$ for all $t\ge1$ and $\mu(B^\prime)=1$, such that for any $x\in B^\prime$,
\begin{gather}
\lambda(x,v)=\lim_{t\to\infty}\frac{1}{t}\log\|\bA(t,x)v\|<0\quad \forall v\in V^s(x)\label{eq2.11}\\
\lambda(x,v)=\lim_{t\to\infty}\frac{1}{t}\log\|\bA(t,x)v\|\ge0\quad \forall v\in \mathbb{R}^d\setminus V^s(x),\label{eq2.12}
\end{gather}
and if $V^s(x)\not=\mathbb{R}^d$ then
\begin{subequations}\label{eq2.13}
\begin{gather}
\limsup_{t\to\infty}\|\bA(t,x)v\|>0\qquad \forall v\in \mathbb{R}^d\setminus V^s(x),\label{eq2.13a}\\
\limsup_{t\to\infty}\pmb{\|}\bA(t,x)\pmb{\|}\ge1.\label{eq2.13b}
\end{gather}\end{subequations}
\end{thm}

\begin{proof}
First it easily follows, from Oselede\v{c}'s multiplicative ergodic theorem, that one can find an $\mathscr{F}$-set $B^\prime\subset X$ with $T^t(B^\prime)\subseteq B^\prime$ for all $t\in \mathbb{Z}_+$ and $\mu(B^\prime)=1$ such that there exists an invariant measurable function
\begin{equation}
B^\prime\ni x\mapsto E^s(x)\in\mathscr{G}(\mathbb{R}^d)
\end{equation}
with the properties:
\begin{gather*}
\lambda(x,v)=\lim_{t\to\infty}\frac{1}{t}\log\|\bA(t,x)v\|<0\quad \forall v\in E^s(x)\\
\intertext{and}
\lambda(x,v)=\lim_{t\to\infty}\frac{1}{t}\log\|\bA(t,x)v\|\ge0\quad \forall v\in \mathbb{R}^d\setminus E^s(x).
\end{gather*}
From Lemma~\ref{lem2.8} and (\ref{eq2.10}), it follows that for all $x\in B^\prime$,
\begin{gather}
\lim_{t\to\infty}\pmb{\|}\bA(t,x)|V^s(x)\pmb{\|}=0\label{eq2.15}\\
\intertext{and}E^s(x)\subseteq V^s(x).
\end{gather}
Then from Theorem~\ref{thm2.7} with $f_n(x)=\log\pmb{\|}\bA(t,x)|V^s(x)\pmb{\|}$ for all $x\in X$ and $n\ge1$, it follows that
for $\mu$-a.e. $x\in X$,
\begin{equation*}
\lim_{t\to\infty}\frac{1}{t}\log\pmb{\|}\bA(t,\omega)|V^s(x)\pmb{\|}<0.
\end{equation*}
Therefore $E^s(x)=V^s(x)$ for $\mu$-a.e. $x\in X$. This proves (\ref{eq2.13a}).

The property (\ref{eq2.13b}) follows from (\ref{eq2.13a}) and Theorem~\ref{thm2.7} with $f_n(x)=\log\pmb{\|}\bA(t,x)\pmb{\|}$ and $X=\{x\colon V^s(x)\not=\mathbb{R}^d\}$.

This thus completes the proof of Theorem~\ref{thm2.9}.
\end{proof}

The property (\ref{eq2.13}) of Theorem~\ref{thm2.9} shows that over almost every driving points $x$, for any nonzero initial state $v_0\not\in V^s(x)$, the state trajectory $\bA(t,x)v_0$ would be far away from the equilibrium $0$ as time $t$ passes.

\subsection{Finer filtration of Matrix-valued cocycles}\label{sec2.4}

To prove the property (f)-(ii) of Theorem~\ref{thm1.1}, we need the following lemma.

\begin{lem}
Let $T$ be a measure-preserving transformation of $(X,\mathscr{F},\mu)$ and $h\colon X\rightarrow\mathbb{R}$ be a random variable.
If $h\circ T\le h$ a.e. then $h\circ T=h$ a.e.
\end{lem}

\begin{proof}
If the statement fails, there is a rational $r$ with $\mu(\{x\colon h>r>h\circ T\})>0$. Then we have $\mu(\{x\colon h>r\})>\mu(\{x\colon h\circ T>r\})$, but these measures are equal since $T$ is measure-preserving and $T^{-1}(\{x\colon h>r\})\subseteq\{x\colon h\circ T>r\}$, a contradiction, proving the assertion.
\end{proof}

As a result of this lemma, we can easily get the following.

\begin{cor}\label{cor2.11}
Under the same situation of Theorem~\ref{thm2.1}, $\dim V^s(x)=\dim V^s(n\bdot x)$ for $\mu$-a.e. $x\in X$ and for any $n>0$.
\end{cor}

Now we are ready to prove our Theorem~\ref{thm1.1}. We will first prove the following elementary version except the item (d)-(iii).

\begin{thm}\label{thm2.12}
Let $T$ be a measure-preserving transformation of the probability space $(X,\mathscr{F},\mu)$ and $\bA\colon\mathbb{Z}_+\times X\rightarrow\mathbb{R}^{d\times d}$ measurable such that $\log^+\pmb{\|}\bA(1,\cdot)\pmb{\|}\in L^1(\mu)$. Then there exists a set $B\in\mathscr{F}$ with $T(B)\subseteq B$ and $\mu(B)=1$ such that:
\begin{enumerate}
\item[$(\mathrm{a})$] There is a measurable function $s\colon B\rightarrow\mathbb{N}$ with $s\circ T=s$.

\item[$(\mathrm{b})$] If $x$ belongs to $B$ there are $s(x)$ numbers $-\infty=\lambda_1(x)<\lambda_2(x)<\dotsm<\lambda_{s(x)}(x)<\infty$.

\item[$(\mathrm{c})$] There are measurable linear subspaces of $\mathbb{R}^d$:
\begin{equation*}
\varnothing=V^{(0)}(x)\subset V^{(1)}(x)\subset\dotsm\subset V^{(s(x))}(x)=\mathbb{R}^d\quad \forall x\in B.
\end{equation*}
\item[$(\mathrm{d})$] If $x$ belongs to $B$, then
\begin{enumerate}
\item[$(\mathrm{i})$] for $1\le i\le s(x)$,
\begin{equation*}
\lim_{n\to\infty}\frac{1}{n}\log\|\bA(n,x)v\|=\lambda_i(x)\quad \forall v\in V^{(i)}(x)\setminus V^{(i-1)}(x);
\end{equation*}
\item[$(\mathrm{ii})$] for $2\le i\le s(x)$,
\begin{gather*}
\limsup_{n\to\infty}e^{-\lambda_i(x)n}\|\bA(n,x)v\|>0\quad \forall v\in V^{(i)}(x)\setminus V^{(i-1)}(x)\\
\intertext{and}
\limsup_{n\to\infty}e^{-\lambda_i(x)n}\pmb{\|}\bA(n,x)|V^{(i)}(x)\pmb{\|}\ge1.
\end{gather*}
\end{enumerate}

\item[$(\mathrm{e})$] The function $\lambda_i(x)$ is defined and measurable on $\{x\,|\,s(x)\ge i\}$ and $\lambda_i^{}(T(x))=\lambda_i(x)$ on this set.

\item[$(\mathrm{f})$] For any $x\in B$ and all $1\le i\le s(x)$,
\begin{enumerate}
\item[$(\mathrm{i})$] $\bA(1,x)V^{(i)}(x)\subseteq V^{(i)}(T(x))$ and
\item[$(\mathrm{ii})$] $\dim V^{(i)}(T(x))=\dim V^{(i)}(x)$.
\end{enumerate}
\end{enumerate}
\end{thm}

\begin{proof}
Let $-\infty=\lambda_1(x)<\lambda_2(x)<\dotsm<\lambda_{s(x)}(x)<\infty$ be the Lyapunov exponents of $\bA$ at $x\in B$ in the sense of Oselede\v{c}'s multiplicative ergodic theorem.

First, by applying Theorem~\ref{thm2.9} and Corollary~\ref{cor2.11} to the $\lambda_{s(x)}(x)$-weighted cocycle
$$\bA^{(s(x))}(n,x)=e^{-\lambda_{s(x)}n}\bA(n,x)$$
driven still by $T$, we can see that for $\mu$-a.e. $x\in B$, the property (d)-(ii) of Theorem~\ref{thm2.12} holds for $i=s(x)$, if $s(x)\ge2$.

Next for $\bA^{(s(x)-1)}(n,x)$ restricted to $V^{(s(x)-1)}(x)$, by the same argument we can see that the property (d)-(ii) of Theorem~\ref{thm2.12} holds for $i=s(x)-1$, if $s(x)\ge3$.

Repeating the above argument completes the proof of Theorem~\ref{thm2.12}.
\end{proof}

To prove the item (d)-(iii) of Theorem~\ref{thm1.1}, we need to use Froyland, LLoyd and Quas~\cite[Theorem~4.1]{FLQ} to obtain the following, in which the property (d)-(iii) is the main point.

\begin{thm}\label{thm2.13}
Let $T$ be a measure-preserving invertible transformation of a Polish probability space $(X,\mathscr{F},\mu)$ and assume $\bA\colon\mathbb{Z}_+\times X\rightarrow\mathbb{R}^{d\times d}$ is measurable such that $\log^+\pmb{\|}\bA(1,\cdot)\pmb{\|}\in L^1(\mu)$. Then there exists a set $B\in\mathscr{F}$ with $T(B)=B$ and $\mu(B)=1$ such that:
\begin{enumerate}
\item[$(\mathrm{a})$] There is a measurable function $s\colon B\rightarrow\mathbb{N}$ with $s\circ T=s$.

\item[$(\mathrm{b})$] If $x$ belongs to $B$ there are $s(x)$ numbers $-\infty=\lambda_1(x)<\lambda_2(x)<\dotsm<\lambda_{s(x)}(x)<\infty$.

\item[$(\mathrm{c})$] There are measurable decompositions of $\mathbb{R}^d$ into linear subspaces:
\begin{equation*}
\mathbb{R}^d=E^{(1)}(x)\oplus\dotsm\oplus E^{(s(x))}(x)\quad \forall x\in B,
\end{equation*}
where $E^{(1)}(x)=\{0\}$ may be permitted.

\item[$(\mathrm{d})$] If $x$ belongs to $B$, then
\begin{enumerate}
\item[$(\mathrm{i})$] for $i=1$, $\lim_{n\to\infty}\frac{1}{n}\log\|\bA(n,x)v\|=\lambda_1(x)$ for all $v\in E^{(1)}(x)$;
\item[$(\mathrm{ii})$] for $2\le i\le s(x)$,
$\lim_{n\to\infty}\frac{1}{n}\log\|\bA(n,x)v\|=\lambda_i(x)$ for all $v(\not=0)\in E^{(i)}(x)$;
\item[$(\mathrm{iii})$] for $2\le i\le s(x)$, one can find some $v_i\in E^{(i)}(x)\setminus\{0\}$ such that
\begin{equation*}
\limsup_{n\to\infty}e^{-\lambda_i(x)n}\|\bA(n,x)v_i\|\ge\|v_i\|.
\end{equation*}
\end{enumerate}

\item[$(\mathrm{e})$] The function $\lambda_i(x)$ is defined and measurable on $\{x\,|\,s(x)\ge i\}$ and $\lambda_i^{}(T(x))=\lambda_i(x)$ on this set.

\item[$(\mathrm{f})$] For any $x\in B$ and all $1\le i\le s(x)$,
\begin{enumerate}
\item[$(\mathrm{i})$] $\bA(1,x)E^{(i)}(x)\subseteq E^{(i)}(T(x))$ and
\item[$(\mathrm{ii})$] $\dim E^{(i)}(T(x))=\dim E^{(i)}(x)$.
\end{enumerate}
\end{enumerate}
\end{thm}

\begin{proof}
Since $(X,\mathscr{F},\mu)$ be a Polish probability space, there is no loss of generality in assuming $T$ is ergodic.
Based on Theorem~\ref{thm2.12} proved above and the improved multiplicative ergodic theorem of Froyland, LLoyd and Quas~\cite{FLQ}, we only need to prove the property (d)-(iii) of Theorem~\ref{thm2.13}. For that, there is no loss of generality is assuming that there is an invariant linear subbundle of $X\times\mathbb{R}^d$
$$
E^c=\bigsqcup_{x\in X}E_x^c
$$
such that $x\mapsto E_x^c$ is measurable, $\dim E_x^c\equiv k$, and for $\mu$-a.e. $x\in X$
\begin{equation*}
\lambda(v):=\lim_{n\to\infty}\frac{1}{n}\log\|\bA(n,x)v\|=0,\quad \forall v\in E_x^c\setminus\{0\}.
\end{equation*}
To prove Theorem~\ref{thm2.13}, it is sufficient to prove that for $\mu$-a.e. $x\in X$, there is a vector $v=v(x)$ in $E_x^c\setminus\{0\}$ such that
\begin{equation*}
\limsup_{n\to\infty}\|\bA(n,x)v\|\ge\|v\|.
\end{equation*}

For that we let $S^{k-1}=\bigsqcup_{x\in X}S_x^{k-1}$ where $S_x^{k-1}:=\left\{v\in E_x^c\colon\|v\|=1\right\}$ and then define a random dynamical system on this random unit sphere bundle
\begin{equation*}
   F\colon S^{k-1}\rightarrow S^{k-1};\quad (x,v)\mapsto\left(Tx, \frac{\bA(1,x)v}{\|\bA(1,x)v\|}\right)
\end{equation*}
driven by the ergodic metric system $(X,\mu,T)$.

For $\mu$-a.e. $x\in X$ and for any $v\in S_x^{k-1}$, $\lambda(v)=0$ implies that $\bA(1,x)v$ is not equal to $0$. Hence $F$ is well defined such that $F_x\colon S_x^{k-1}\rightarrow S_{Tx}^{k-1}$ is (linear) continuous with respect to $v\in S_x^{k-1}$, for $\mu$-a.e. $x\in X$. We need to note that since $Tx$ is only measurable in $x$, $F$ is not necessarily continuous on the bundle $S^{k-1}$.

Let $\mathcal{I}_\mu(F)$ be the set of all $F$-invariant Borel probability measures on $S^{k-1}$ covering $\mu$ by the natural projection $\pi\colon (x,v)\mapsto x$ from $S^{k-1}$ onto $X$.
By the standard theorem of existence of invariant measures (cf., e.g., \cite[Theorem~1.5.10]{Arn}), $\mathcal{I}_\mu(F)$ is a non-void, compact and convex set.

Since $\bA(1,x)$ is measurable in $x$, we can define a measurable characteristic function
\begin{equation*}
    \varphi\colon S^{k-1}\rightarrow \mathbb{R};\quad (x,v)\mapsto\log\|\bA(1,x)v\|\;\forall (x,v)\in S^{k-1},
\end{equation*}
which is such that
\begin{equation*}
    \log\|\bA(n,x)v\|=\sum_{i=0}^{n-1}\varphi(F^i(x,v))\quad\forall (x,v)\in S^{k-1}.
\end{equation*}
Let $\tilde{\mu}\in\mathcal{I}_\mu(F)$ be arbitrarily given. By $\{\tilde{\mu}_x\}_{x\in X}$ we denote the standard disintegration of $\tilde{\mu}$ given $\mu$ via the projection $\pi$. Since $\varphi^+(x,v)\le\log^+\|\bA(1,x)\|$ for all $(x,v)\in S^{k-1}$, we have
\begin{equation*}
\int_{S^{k-1}}\varphi^+d\tilde{\mu}=\int_X\left(\int_{S_x^{k-1}}\varphi^+(x,v)d\mu_x(v)\right)d\mu(x)\le\int_X\log^+\|\bA(1,x)\|d\mu(x)<\infty.
\end{equation*}
Hence $\varphi^+$ belongs to $L^1(\tilde{\mu})$ but $\varphi$ does not need to be in $L^1(\tilde{\mu})$.

Then applying Theorem~\ref{thm2.7} with $f_n(x,v)=\sum_{i=0}^{n-1}\varphi(F^i(x,v))$ for all $(x,v)\in S^{k-1}$, it follows that
\begin{equation*}
    \limsup_{n\to\infty}\log\|\bA(n,x)v\|\ge0\quad\textrm{and thus}\quad \limsup_{n\to\infty}\|\bA(n,x)v\|\ge1\quad\tilde{\mu}\textrm{-a.e. }(x,v)\in S^{k-1}.
\end{equation*}
This completes the proof of Theorem~\ref{thm2.13}.
\end{proof}

We note here that in the proof of Theorem~\ref{thm2.13}, since the characteristic function $\varphi$ is not necessarily to be $\tilde{\mu}$-integrable, we cannot use Atkinson's Lemma~\ref{lem2.6}; see Remark~\ref{rem5}.

Finally, combining Theorems~\ref{thm2.12} and \ref{thm2.13} we can complete the proof of Theorem~\ref{thm1.1}.

\begin{proof}[Proof of Theorem~\ref{thm1.1}]
According to Theorem~\ref{thm2.12}, we only need to prove the item (d)-(iii) of Theorem~\ref{thm1.1}. However, this property can also be easily induced from Theorem~\ref{thm2.13} by using the natural extension of the cocycle $\bA$; see, e.g., \cite[Section~6.2]{Dai-07}. 
\end{proof}

\section{Conditional stability of linear random processes}\label{sec3}
In this section, we shall give an application of Theorem~\ref{thm1.1} to the study of conditional stability of linear random system.

Suppose $\X=\{x_{\bdot}\colon\mathbb{Z}_+\rightarrow \bS\}$ is the Cartesian product $\bS^{\mathbb{Z}_+}$ of a fixed measurable space $(\bS,\mathcal{S})$. Here $\X$ possesses a natural $\sigma$-algebra $\mathcal{S}^{\mathbb{Z}_+}$ generated by cylindrical sets of the form
\begin{equation}\label{eq3.1}
A=\left\{x_{\bdot}\in \X\,|\,x_{i_1}\in C_1,\dotsc,x_{i_r}\in C_r\right\}
\end{equation}
where $1\le r<\infty, 0\le i_1<\dotsm<i_r<\infty$ are integers and $C_1,\dotsc,C_r\in\mathcal{S}$. Suppose $\mu$ be a probability measure on $\mathcal{S}^{\mathbb{Z}_+}$ and $\mathcal{S}_\mu^{\mathbb{Z}_+}$ is the completion of $\mathcal{S}^{\mathbb{Z}_+}$ with respect to $\mu$. In probability theory the triple $\bxi=(\X,\mathcal{S}_\mu^{\mathbb{Z}_+},\mu)$ is said to be a \textit{discrete-time random process}, where $\X$ is the sample-path space and $\bS$ the state space of this process.

If, for any set $A$ of the form (\ref{eq3.1}), the measure $\mu(\{x_{\bdot}\in \X\,|\,x_{i_1+n}\in C_1,\dotsc,x_{i_r+n}\in C_r\})$ does not depend upon $n$, $0\le n<\infty$, then the process $\bxi$ is called \textit{stationary}. Let us express the stationary condition in another way. Define the shift transformation
\begin{equation}\label{eq3.2}
T\colon \X\rightarrow \X;\quad x_{\bdot}=(x_0,x_1,x_2,\dotsc)\mapsto x_{\bdot+1}=(x_1,x_2,x_3,\dotsc).
\end{equation}
Then if $A$ is a set of the form (\ref{eq3.1}) we have $T^{-1}(A)=\{x_{\bdot}\in \X\,|\,x_{i_1+1}\in C_1,\dotsc,x_{i_r+1}\in C_r\}$, and the stationarity condition may be written in the form $\mu(T^{-1}(A))=\mu(A)$. Since $\mu$ is uniquely determined by its values on cylindrical sets, stationarity condition means that the shift transformation $T$ preserves $\mu$, i.e., $T$ is a measure-preserving transformation of the probability space $(\X,\mathcal{S}_\mu^{\mathbb{Z}_+},\mu)$.

The followings are three important stationary random processes which often serve as our driving dynamical systems.

\begin{example}[Bernoulli process]\label{exa3.1}
Let $\bxi$ be the Cartesian product
$$
(\X,\mathcal{S}_\mu^{\mathbb{Z}_+},\mu)=\prod_{n=0}^{\infty}(S_n,\mathcal {S}_n,\varrho_n),
$$
where $(S_n,\mathcal {S}_n,\varrho_n)=(\bS,\mathcal {S},\varrho)$ is a probability space. The measure $\mu=\bigotimes_{n=0}^\infty \varrho_n$ is the countable product-measure generated by the measure $\varrho$. Then the shift transformation $T$ is ergodic and mixing (cf.~\cite[Theorem~8.1]{CFS}).
\end{example}

\begin{example}[Markov process]\label{exa3.2}
A \textit{stochastic operator} on the state space $(\bS,\mathcal{S})$ is a function $\P(s,C)$ of the variables $s\in\bS, C\in\mathcal{S}$ with the following properties:
\begin{enumerate}
\item[(1)] $\P(s,\bdot)$, for any fixed $s\in \bS$, is a probability measure on the measurable space $(\bS,\mathcal{S})$;
\item[(2)] $\P(s,C)$, for any fixed $C\in\mathcal{S}$, is a measurable function on $\bS$.
\end{enumerate}
A probability measure $\nu$ on $(\bS,\mathcal{S})$ is said to be an \textit{invariant measure for the stochastic operator} $\P$ if for any $C\in\mathcal{S}$ we have
$$
\nu(C)=\int_{\bS}\P(s,C)d\nu(s).
$$
Given a stochastic operator $\P$ and an invariant probability measure $\nu$, we can define a measure $\mu_{\nu,\P}$ on the sample-path space $(\X,\mathcal{S}^{\mathbb{Z}_+})$ in the following way: First for the cylindrical sets
$$
A=\left\{x_{\bdot}\in \X\,|\,x_{i}\in C_0,x_{i+1}\in C_1,\dotsc,x_{i+r}\in C_r\right\}, \quad \textrm{where }r\ge 0, i\ge0, C_0,\dotsc,C_r\in\mathcal{S},
$$
we set
$$
\mu_{\nu,\P}(A)=\int_{C_0}d\nu(x_i)\int_{C_1}\P(x_i, dx_{i+1})\dotsm\int_{C_r}\P(x_{i+r-1},dx_{i+r}).
$$
Then, using the Kolmogorov extension theorem, uniquely extend $\mu_{\nu,\P}$ to the entire $\sigma$-algebra $\mathcal{S}^{\mathbb{Z}_+}$ and then to $\mathcal{S}_{\mu_{\nu,\P}}^{\mathbb{Z}_+}$. The invariance of $\nu$ implies that the probability measure $\mu_{\nu,\P}$, which is called a Markovian measure, is stationary. In this case, $\bxi=(\X,\mathcal{S}_{\mu_{\nu,P}}^{\mathbb{Z}_+},\mu_{\nu,\P})$ is called a \textit{Markov process}.

We should note that the ergodic properties of a Markov process $\bxi$ may differ for various initial distribution $\nu$ and stochastic operator $\P$. It is easy to construct examples of non-ergodic Markov processes even if the state space $\bS$ is finite.
\end{example}

Let $\mathscr{B}_{\mathbb{R}}$ be the standard Borel $\sigma$-algebra of $\mathbb{R}$. The following is just the discretization of the classical stationary $1$D-Brownian process.

\begin{example}[$\mathrm{1D}$-Brownian motion]\label{exa3.3}
Let $\X=\mathbb{R}^{\mathbb{Z}_+}$ and $\mathscr{B}=\mathscr{B}_{\mathbb{R}}^{\mathbb{Z}_+}$. We now define a stochastic operator $\P$ on $(\mathbb{R},\mathscr{B}_{\mathbb{R}})$ as follows: For any $y\in\mathbb{R}$ and $C\in\mathscr{B}_\mathbb{R}$, let
$$
P(y,C)=\int_C\frac{1}{\sqrt{2\pi}}e^{-\frac{(z-y)^2}{2}}dz.
$$
Let $\nu$ be a probability measure $(\mathbb{R},\mathscr{B}_\mathbb{R})$, which is invariant for $\P$. Then as in Example~\ref{exa3.2}, we can get a Markovian measure $\mu_{\nu,\P}$. In this case, $\bxi=(\mathbb{R}^{\mathbb{Z}_+},\mathscr{B},\mu_{\nu,\P})$ is called a discrete-time 1-dimensional \textit{Brownian motion}.
\end{example}

From now on, we let $A\colon \bS\rightarrow\mathbb{R}^{d\times d}$ be a matrix-valued measurable function, which is bounded, i.e., $\pmb{\|}A(s)\pmb{\|}\le\beta$ for all $s\in\bS$ for some constant $\beta$. Then based on a stationary random process $\bxi=(\X,\mathcal{S}_\mu^{\mathbb{Z}_+},\mu)$ with the state space $(\bS,\mathcal{S})$, it gives rise to a linear random system:
\begin{equation}\label{eq3.3}
\bA_\bxi\colon \mathbb{Z}_+\times\X\rightarrow \mathbb{R}^{d\times d};\quad (n,x_{\bdot})\mapsto\begin{cases}A(x_{n-1})\dotsm A(x_0)& \textrm{if }n\ge1,\\ I_{d\times d} & \textrm{if }n=0.\end{cases}
\end{equation}
It is just a linear cocycle driven by the shift transformation $T$ as in (\ref{eq3.2}).

We will consider the following two kinds of stability of $\bA_{\bxi}$, which may be regarded as the random versions of \cite{SU94, CRS99}.

\begin{defn}\label{def3.4}
Let $\mathbb{L}$ be a linear subspace of $\mathbb{R}^d$. The linear random system $\bA_\bxi$ is said to be:
\begin{itemize}
\item \textit{$\mathbb{L}$-conditionally Lyapunov stable}, if $\mu(\left\{x_{\bdot}\in \X\,|\,\lim_{n\to\infty}\pmb{\|}\bA_\bxi(n,x_{\bdot})|\mathbb{L}\pmb{\|}=0\right\})>0$;
\item \textit{$\mathbb{L}$-conditionally exponentially stable}, if $\mu(\left\{x_{\bdot}\in \X\,|\,\lim_{n\to\infty}n^{-1}\log\pmb{\|}\bA_\bxi(n,x_{\bdot})|\mathbb{L}\pmb{\|}<0\right\}>0$.
\end{itemize}
\end{defn}

These two types of stability seem, at the first glance, to be different from each other even in the $1$-dimensional case as is shown by Example~\ref{exa2.5} in Section~\ref{sec2}.

Conceptually the conditional Lyapunov stability of $\bA_\bxi$ is easier to check than the conditional exponential stability; but the latter is more popular than the former in the theory of multi-rate sampled-data control systems, multi-modal linear control systems, numerical calculus, and for some control optimization problems; see, for example, \cite{Stan79, BCS88, SU94, GR97, CRS99, Sun04, BD, Sun06, LD06, LD07, Dai-JDE} and so on. An explicit simple example of application is given as follows:

Let $V\colon \mathbb{R}^d\rightarrow[0,\infty)$ be a continuous function, which is locally Lipschitz at the origin zero, that is, $V(u)\le\gamma\|u\|$ for all $u\in\mathbb{R}^d$ with $\|u\|\le\delta$, for some $\delta>0$. Associated to $V$ we consider the infinite-time cost index of $(\bxi,A)$ given on $\bS$ by
\begin{equation*}
\mathscr{L}(u,x_{\bdot})=\sum_{n=0}^{+\infty}V(\bA_\bxi(n,x_{\bdot})u),\quad \forall u\in \mathbb{L}\textrm{ and }x_{\bdot}\in \X.
\end{equation*}
Because
$$0\le \mathscr{L}(u,x_{\bdot})\le\Delta_{x_{\bdot}}+\gamma\|u\|\sum_{n=N_{x_{\bdot}}}^{+\infty}\pmb{\|}\bA_\bxi(n,x_{\bdot})|\mathbb{L}\pmb{\|},$$
if $\bA_\bxi$ is $\mathbb{L}$-conditionally exponentially stable $\mathscr{L}(u,x_{\bdot})$ is finite for some $x_{\bdot}$ of $\mu$-positive measure. Then we can study the \textit{optimal cost} of $\bxi$ associated to $\mathscr{L}$ at $u\in \mathbb{L}$ that may be defined as $J(u)=\mu\textrm{-ess.}\inf_{x_{\bdot}\in\X}\mathscr{L}(u,x_{\bdot})$.

So how to characterize the exponential stability of $\bA_\bxi$ from the Lyapunov stability has become more and more interesting recently.
Now we will prove the following equivalent relationship using Theorem~\ref{thm1.1}, which generalizes \cite[Theorem~B$^\prime$]{Dai-JDE}.

\begin{thm}\label{thm3.5}
Given any linear subspace $\mathbb{L}\subseteq \mathbb{R}^d$ and based on a stationary random process $\bxi$, $\bA_\bxi$ is $\mathbb{L}$-conditionally Lyapunov stable if and only if $\bA_\bxi$ is $\mathbb{L}$-conditionally exponentially stable.
\end{thm}

\begin{proof}
We need only prove the necessity. Assume $\bA_\bxi$ is $\mathbb{L}$-conditionally Lyapunov stable; i.e., if write
$$
\varLambda=\left\{x_{\bdot}\in \X\,|\,{\lim}_{n\to\infty}\pmb{\|}\bA_\bxi(n,x_{\bdot})|\mathbb{L}\pmb{\|}=0\right\}
$$
then $\mu(\varLambda)>0$. For $\mu$-a.e. $x_{\bdot}\in\varLambda$, let $-\infty=\lambda_1(x_{\bdot})<\dotsm<\lambda_{r(x_{\bdot})}(x_{\bdot})<\dotsm<\lambda_{s(x_{\bdot})}(x_{\bdot})<\infty$ be the Lyapunov exponents of $\bA_\bxi$ at the base point $x_{\bdot}$ given by Theorem~\ref{thm1.1}.

If $\lambda_{s(x_{\bdot})}(x_{\bdot})$ is less than $0$, then $\lim_{n\to\infty}n^{-1}\log\pmb{\|}\bA_\bxi(n,x_{\bdot})|\mathbb{L}\pmb{\|}<0$ from the property (d)-(i) of Theorem~\ref{thm1.1}. So from now on, without loss of generality we may assume $\lambda_{s(x_{\bdot})}(x_{\bdot})\ge0$ and let $\lambda_{r(x_{\bdot})}(x_{\bdot})<0\le\lambda_{r(x_{\bdot})+1}(x_{\bdot})$ for $\mu$-a.e. $x_{\bdot}\in\varLambda$. Then the property (d)-(ii) of Theorem~\ref{thm1.1} implies that $\mathbb{L}\subseteq V^{(r(x_{\bdot}))}(x_{\bdot})$ for $\mu$-a.e. $x_{\bdot}\in\varLambda$.
This thus completes the proof of Theorem~\ref{thm3.5}.
\end{proof}

The most interesting case of Theorem~\ref{thm3.5} is that $\bxi$ is a Markov process or a Brownian motion in the theory of control and optimizations.

We note that because $\mathbb{L}\subset\mathbb{R}^d$ is not necessarily $\bA_\bxi$-invariant, $f_n(x_{\bdot})=\log\pmb{\|}\bA_\bxi(n,x_{\bdot})|\mathbb{L}\pmb{\|}$ is not necessarily to be subadditive on $\X$ with respect to the shift transformation $T\colon \X\rightarrow \X$ and then we can not directly employ Theorem~\ref{thm2.7} here. In addition for many control optimization problems $\mathbb{L}\neq\mathbb{R}^d$ because of constraint conditions.


\subsection*{\textbf{Acknowledgments}}%
This work was supported in part by National Science Foundation of China Grant $\#$11271183, PAPD of Jiangsu Higher Education Institutions.

The author would like to thank the anonymous referees for many helpful suggestions.

\end{document}